\documentclass[12pt]{amsart}
\usepackage{amssymb,epsfig,amsmath,latexsym,amsthm}
\usepackage{graphicx}
\usepackage{enumitem}
\usepackage{comment}
\usepackage[hidelinks]{hyperref}
\theoremstyle{plain}
\newtheorem{theorem}{Theorem}[section]
\newtheorem{lemma}[theorem]{Lemma}
\newtheorem{claim}[theorem]{Claim}
\newtheorem{proposition}[theorem]{Proposition}

\theoremstyle{definition}

\newtheorem{remark}[theorem]{Remark}

\usepackage{epsfig,color}

\makeatother

\theoremstyle{definition}

\def\fnum{equation}

\numberwithin{equation}{section}

\usepackage[hidelinks]{hyperref}

\begin{document}
\title
[Rigidity for the equatorial  disk in the ball]
{Area Rigidity for the  equatorial  disk in the ball}

\author{Ezequiel Barbosa}
\address{Universidade Federal de Minas Gerais (UFMG), Caixa Postal 702, 30123-970, Belo Horizonte, MG, Brazil}

\email{ezequiel@mat.ufmg.br}
\author{Celso Viana}
\address{Universidade Federal de Minas Gerais (UFMG), Caixa Postal 702, 30123-970, Belo Horizonte, MG, Brazil}
\email{celso@mat.ufmg.br}

\begin{abstract}
It is proved by Brendle in \cite{B} that the equatorial disk $D^k$ has least area among  $k$-dimensional free boundary minimal surfaces in the Euclidean ball $B^n$. By comparing the excess of free boundary minimal surfaces    with the excess of the associated  cones over their boundary, we prove  the existence of a gap for the area.
\end{abstract}

\maketitle

\section{Introduction}\label{intro}
\noindent In these notes, we study the area  of $k$-dimensional minimal surfaces in the Euclidean  ball $B^n$ that meet  $\partial B^n$ orthogonally.  These surfaces  are critical points of the area functional in the space of  $k$-dimensional surfaces with   boundary  in $\partial B^n$. They are commonly known as free boundary minimal surfaces. The  equatorial disk $D^k$  is the simplest example. Brendle \cite{B} proved that $D^k$ is the least area   free boundary minimal surface in $B^n$ (see also \cite{FS1} for the case of $2$-dimensional free boundary surfaces). More precisely,

\begin{theorem}[Brendle]\label{brendle}
\textit{Let $\Sigma^k$ be a $k$-dimensional free boundary minimal surface in $B^n$. Then 
 \begin{eqnarray*}
|\Sigma^k|\geq |D^k|
\end{eqnarray*}
Moreover, the equality holds if, and only if,  $\Sigma^k$ is contained in a $k$-dimensional plane in $\mathbb{R}^{n}$.
}
\end{theorem}
This result is the free boundary analogue of a classical result about closed minimal surfaces in the  round sphere $\mathbb{S}^{n}$. Namely,
\begin{theorem}\label{almgren}
\textit{
There exists $\varepsilon(k,n)>0$ so that whenever $\Sigma^k$ is a $k$-dimensional minimal surface in $\mathbb{S}^{n}$ which is not totally geodesic, then
\[
|\Sigma^k|\geq |\mathbb{S}^k|+ \varepsilon(k,n).
\]
}
\end{theorem}
Despite the proofs of Theorem 1.1 and Theorem 1.2   both explore a  monotonicity principle for minimal surfaces, they are quite different.  Theorem 1.2, for instance, is only an application of the  Monotonicity Formula for minimal surfaces together with the following    smooth version of Allard's Regularity Theorem: 
\begin{theorem}[Allard]
There exist $\epsilon(k,n)>0$, $C>0$ and $r_0>0$ so that whenever $\Sigma$ is a $k-$dimensional minimal surface in $\mathbb{R}^{n+1}$ whose density satisfies 
\[\theta(x,r)\leq 1+\epsilon(k,n)\]
for every $x\in\Sigma$ and every $r<r_0$, then
\[\sup_{\Sigma}|A_{\Sigma}|\leq C.\]
\end{theorem}
\noindent Indeed,
let $\Sigma_i$ be a sequence of  $k-$dimensional minimal surfaces in $\mathbb{S}^{n}$ such that $\lim_{i\rightarrow\infty}|\Sigma_i|=\mathcal{A}(k,n)$, where $\mathcal{A}(k,n)$ is the infimum for the areas of free boundary minimal surfaces in $\mathbb{S}^n$. If $C\Sigma_i$ denotes  the  minimal cone over $\Sigma_i$   with vertice at $0$ and if $y_i \in \Sigma_i$, then
\[\frac{|\Sigma_i|}{|\mathbb{S}^k|}= \lim_{r\rightarrow\infty} \frac{|C\Sigma_i\cap B_r(y_i)|}{|B^{k+1}|r^{k+1}}\geq \frac{|C\Sigma_i\cap B_r(y_i)|}{|B^{k+1}|r^{k+1}}= \theta(C\Sigma_i,y_i,r)\geq 1,\]
with equality if, and only if, $\Sigma_i$ is an equatorial sphere $\mathbb{S}^k$.
The inequality follows from the monotonicity formula for minimal surfaces. Hence, $\mathcal{A}(k,n)= |\mathbb{S}^k|$ and from Theorem 1.3 we conclude that  $|A_{\Sigma_i}|\leq C$. By  standard compactness results, $\Sigma_i$ converges graphically  and with multiplicity one to $\mathbb{S}^k$. 
A comparison analysis between the Morse index  of $\Sigma_i$ and $\mathbb{S}^k$ implies that $\Sigma_i$ is an equatorial sphere for $i$ large enough, see Section $3$ below.

In view of Theorems \ref{brendle} and \ref{almgren}, it is natural to expect similar gap phenomena    also for the area of free boundary  minimal surfaces in $B^n$. In contrast with Theorem 1.2, the  smooth free boundary version of Allard's regularity theorem  does not readily apply to this end. 
It can be  proved, however, that it follows from the strong  Allard's regularity theorem, proved by  Gr\"{u}ter and Jost \cite{GJ}, together with the analysis developed in \cite{B}, which we also use here. 
Our first result is a direct and simpler proof of this fact:

\begin{theorem}\label{high dimension}
There exists $\varepsilon(k,n)>0$ such that whenever $\Sigma^k$ is a $k$-dimensional free boundary minimal surface in $B^n$ satisfying
\begin{eqnarray*}
|\Sigma^k|\,< \,|D^k|+ \varepsilon(k,n),
\end{eqnarray*} 
then $\Sigma^k$ is, up to ambient isometries, the equatorial disk $D^{k}$.
\end{theorem}
The $2$-dimensional case in Theorem \ref{high dimension}  was  proved by Ketover  in \cite{K}.
The key ingredients in the proof are an excess inequality for $2$-dimensional free boundary  surfaces in  $B^n$, proved  by Vokmann in \cite{V} (see also \cite{R-V}), and the classical Nitsche's Uniqueness Theorem for free boundary minimal disks in $B^3$ (see also \cite{FS}, for the generalization  to  high codimension). The  excess inequality  is particularly important  in proving curvature estimates for a sequence of free boundary minimal surfaces with area sufficiently close to the area of the equatorial disk. The   main difficulty in implementing the arguments of \cite{K} to  $k$-dimensional surfaces in $B^n$ is   that  neither the excess inequality in the form used in \cite{K} nor  Nitsche's  Theorem is  readily available when $k\geq 3$. To get around these difficulties, we consider a slightly more general quantity, originated in \cite{B} and   which also resemble an excess type formula, and compare it with that of the free boundary cones over the boundaries  to  obtain the necessary curvature estimates. Finally, we replace the use of Nitsche's   Theorem by an index of stability analysis.

\begin{remark}
We observe that the $2$-dimensional proof of Theorem \ref{high dimension} given in \cite{K} can be extended  to constant mean curvature surfaces in $B^3$. 
The   quantity to consider  in this case is the Willmore energy instead of    area.	
Let $\Sigma^2$ be a  surface with boundary  in $\mathbb{R}^3$,  the Willmore energy $\mathcal{W}(\Sigma)$  is defined as 
	\begin{eqnarray*}\label{willmore energy}
	\mathcal{W}(\Sigma)\,=\, \int_{\Sigma} H^2\,d_{\Sigma} + \int_{\partial \Sigma} k_g\,d\sigma.
	\end{eqnarray*}
\noindent \textbf{$\varepsilon$-Regularity.}  There exists $\varepsilon>0$  such that whenever  $\Sigma$ is a free boundary  surface with constant mean curvature in $B^3$ and satisfying
	\begin{eqnarray*}
		\mathcal{W}(\Sigma)\,<\, 2\pi + \varepsilon,
	\end{eqnarray*}
	then $\Sigma$ is either an equatorial disk or a spherical cap. The constant $\varepsilon$ is independent of the value of the mean curvature.
\end{remark}

Finally, we recall an unique continuation result which might be  of  independent interest in view of the discussion above. This result seems to be well known among experts but not clearly stated in the literature: 
\begin{proposition}\label{unique continuation}	\textit{Let $\Sigma^k$ be a  $k$-dimensional   free boundary minimal surface in $B^n$ which is smooth except possibly at the origin. If $\partial\Sigma$ is a $(k-1)$-minimal surface in $\mathbb{S}^{n-1}$, then $\Sigma^k$ is the minimal cone $C_1\partial \Sigma$.}	\end{proposition}

\section{Higher dimension free boundary minimal surfaces}

\noindent
 We start by recalling an excess inequality  for free boundary  minimal surfaces in the ball proved in \cite{B}.  More precisely,
 if $\Sigma$ is a $k$-dimensional free boundary minimal surface in $B^n$ and if $y \in \partial \Sigma$, then 
\begin{eqnarray}\label{excess bordo}
\int_{\Sigma^k}  \frac{|(x-y)^{\perp}|^2}{|x-y|^{k+2}}\, d_{\Sigma} \leq |\Sigma^k|- |D^k|.
\end{eqnarray}
 This inequality, which implies Theorem \ref{brendle},  follows from a monotonicity argument obtained by an application of the  Divergence Theorem to the vector field $W_{t_0,y}(x)$ defined on $B^n-\{y\}$ and given  by
\[
W_{t_0,y}(x)= \frac{x}{2}- \frac{x-y}{|x-y|^k}-\frac{k-2}{2}\int_{t_0}^{|y|^2}\frac{tx-y}{|tx-y|^k}dt,
\]
where we assume that  $t_0\in\{0,1\}$ if $y\in \partial \Sigma$ and $t_0=|y|^2$ if $y\notin \partial \Sigma$. Note that the   integrand  is well defined by our definition of $t_0$.
We will need a formula similar to (\ref{excess bordo}) for when $y$ is not necessarily at the boundary. For this, we need to recall the techniques  in \cite{B} involved in the proof of (\ref{excess bordo}).
\begin{lemma}\label{excess formula geral}
Let $\Sigma^k$ a free boundary surface in $B^n$ and $y\in \Sigma$. For $r$ sufficiently  small, we have
\begin{eqnarray}\label{excess hypersurface}
	&2&\int_{\Sigma\backslash B_r(y)}\, \frac{|(x- y)^{\perp}|^2}{|x-y|^{k+2}}\,d_{\Sigma} + \frac{k-2}{k}\int_{\Sigma \backslash B_r(y)} \int_{t_0}^{|y|^2} \frac{t\,|(tx- y)^{\perp}|^2}{|tx-y|^{k+2}}\, dt\, d_{\Sigma}\nonumber \\
	\nonumber \\
	&=&|\Sigma \backslash B_r(y)| 
	-\frac{2}{k}\int_{\Sigma\cap \partial B_r(y)} \langle W_{t_0,y}(x),\nu(x)\rangle\,d\sigma  \nonumber \\ 
	\nonumber \\
	&& \quad\quad \quad\quad\quad\quad\quad  - \frac{2}{k} \int_{\partial\Sigma} \langle W_{t_0,y},x\rangle d\sigma +2\int_{\Sigma\backslash B_r(y)} \langle \overrightarrow{H},W_{t_0,y}\rangle d_{\Sigma}, 
\end{eqnarray}
where $B_r(y)$ is a geodesic ball on $\Sigma^k$, $\nu$ is the outward co-normal to $\partial B_r(y)$  with respect to $\Sigma\backslash B_r(y)$, and $\overrightarrow{H}$ is the mean curvature vector of $\Sigma^k$.
\end{lemma}
\begin{proof}
A computation following  Section 2 in \cite{B} gives
\[
div_{\Sigma} W_{t_0,y}= \frac{k}{2} - k \frac{|(x-y)^{\perp}|^2}{|x-y|^{k+2}} - \frac{k-2}{2} \int_{t_0}^{|y|^2} tk \frac{|(tx-y)^{\perp}|^2}{|tx-y|^{k+2}}.
\]
On the other hand, we have that
\[
div_{\Sigma}W_{t_0,y}=  div_{\Sigma}W_{t_0,y}^{\top} - k\langle W_{t_0,y}, \overrightarrow{H} \rangle.
\]
Integrating both sides above over $\Sigma\backslash B_r(y)$ and applying the Divergence Theorem we conclude the proof.
\end{proof}

The next lemma deals with the second term in the right hand side of (\ref{excess hypersurface}):
\begin{lemma}\label{ordem}
Let $\Sigma^k$ be a free boundary  surface in $B^n$ and let $\varphi(y)=1$ if $y\in \partial \Sigma$ and $\varphi(y)=2$ if $y \in \Sigma\backslash \partial \Sigma$. Then
\[\lim_{r\rightarrow 0}\frac{2}{k}\int_{\Sigma\cap \partial B_r(y)} \langle W_{t_0,y}(x),\nu(x)\rangle = \varphi(y)\,|D^k|. \]
\end{lemma} 
\begin{proof}
The proof is essentially contained in Section 2 of \cite{B}. More precisely,  an application of   \cite[Lemma 8]{B} gives
	\[
	W_{t_0,y}(x)= - \frac{x-y}{|x-y|^k} + o(\frac{1}{|x-y|^{k-1}})
	\]
Note that this statement is trivial when $y\notin \partial \Sigma$ or when $y\in \partial \Sigma$ and $t_0=1$. The lemma now follows from the computations leading to equation (2) in \cite{B}; the only minor difference  comes from the possibility of $y\notin \partial \Sigma$.
\end{proof}

\begin{lemma}\label{3}
If $y\in \partial \Sigma$, then $\langle W_{0,y}(x),x\rangle=0$ for every $x\in\partial\Sigma-\{y\}$.
\end{lemma}
\begin{proof}
See Section 2 in \cite{B}.
\end{proof}

\begin{remark}
Applying Lemmas \ref{excess formula geral},  \ref{ordem}, and  \ref{3}, we obtain  the general form of  inequality  (\ref{excess bordo}) when  $y\in \partial \Sigma$:
\begin{equation}
\int_{\Sigma} \frac{|(x-y)^{\perp}|^2}{|x-y|^{k+2}} + \frac{k-2}{2k}\int_{\Sigma} \int_{0}^1\frac{tk|(tx-y)^{\perp}|^2}{|tx-y|^{k+2}} -  \int_{\Sigma} \langle \overrightarrow{H}, W_{0,y} \rangle = \frac{|\Sigma| - |D^k|}{2}.
\end{equation}
\end{remark}

\begin{proposition}\label{main result}
Let $\Sigma^k$ be a $k$-dimensional free boundary minimal surface in $B^n$ and $C_1\partial\Sigma$  the   cone with vertice at the origin and base $\partial\Sigma$. If $y\in \Sigma - C_1\partial \Sigma$, then
		\begin{equation}\label{excess brendle}
		\int_{\Sigma} \frac{|(x-y)^{\perp}|^2}{|x-y|^{k+2}}=	\int_{C_1\partial\Sigma} \frac{|(x-y)^{\perp}|^2}{|x-y|^{k+2}}  + \int_{C_1\partial\Sigma}\langle \overrightarrow{H}_{C_1\partial\Sigma},\frac{x-y}{|x-y|^k}\rangle - |D^k|.
		\end{equation}
If $y\in \partial\Sigma$, then
\begin{equation}\label{excess cone}
	\int_{\Sigma} \frac{|(x-y)^{\perp}|^2}{|x-y|^{k+2}}=	\int_{C_1\partial\Sigma} \frac{|(x-y)^{\perp}|^2}{|x-y|^{k+2}}  + \int_{C_1\partial\Sigma}\langle \overrightarrow{H}_{C_1\partial\Sigma},\frac{x-y}{|x-y|^k}\rangle.
\end{equation}
\end{proposition}

\begin{proof}
For this proposition we choose $t_0=|y|^2$. Hence, the vector field $W_{t_0,y}$ becomes
\begin{eqnarray*}
		W_y=\frac{x}{2}- \frac{x-y}{|x-y|^k}.
\end{eqnarray*}
 First the case $y\notin \partial \Sigma$. Applying
	 Lemma \ref{excess formula geral} and Lemma \ref{ordem} we obtain
	\begin{equation}\label{equation excess}
	2\int_{\Sigma}\, \frac{|(x- y)^{\perp}|^2}{|x-y|^{k+2}}\,d_{\Sigma} =|\Sigma| 
	-2|D^k| 
	-\frac{2}{k} \int_{\partial\Sigma}\langle W_y(x),x\rangle d\sigma. 
	\end{equation}
	Now we look at the last  term in (\ref{equation excess}). Let  $C_1\partial \Sigma$ be the free boundary cone over $\partial \Sigma$ and vertice at $0$.  Applying Lemma \ref{excess formula geral}  to $C_1\partial \Sigma$ and observing that  $C_1\partial \Sigma$ might not be  a minimal surface, we obtain:
	\begin{eqnarray*}\label{excess sequence}
		2\int_{C_1\partial\Sigma\backslash B_r(0)}\, \frac{|(x- y)^{\perp}|^2}{|x-y|^{k+2}} 	=|C_1\partial\Sigma\, \backslash\, B_r(0)| 
		-\,\frac{2}{k}\int_{C_1\partial\Sigma\cap \partial B_r(0)} \langle W_{y},\nu\rangle\,d\sigma  \\
	-\frac{2}{k} \int_{\partial\Sigma}\langle W_y(x),x\rangle d\sigma
		+2\int_{C_1\partial \Sigma \backslash B_r(0)} \langle \overrightarrow{H}_{C_1\partial\Sigma},W_y\rangle\, d_{C_1\partial \Sigma}.
	\end{eqnarray*}
	Taking the limit  as $r\rightarrow 0$ in above expression, we obtain
	\begin{eqnarray}\label{conta}
		2\int_{C_1\partial\Sigma} \frac{|(x- y)^{\perp}|^2}{|x-y|^{k+2}}\,d_{C_1\Sigma}
	= |C_1\partial\Sigma| &-& \frac{2}{k} \int_{\partial\Sigma}\langle W_y(x),x\rangle d\sigma \nonumber \\
	&+& 2\int_{C_1\partial \Sigma} \langle \overrightarrow{H}_{C_1\partial\Sigma},W_y\rangle\, d_{C_1\partial \Sigma}. 
	\end{eqnarray}
	Plugging (\ref{excess sequence}) into (\ref{equation excess}),  we obtain
			\begin{eqnarray}\label{substitui}
	2\int_{\Sigma}\, \frac{|(x- y)^{\perp}|^2}{|x-y|^{k-2}}\,d_{\Sigma} 
  &=&|\Sigma|-|C_1\partial\Sigma|+2\int_{C_1\partial\Sigma} \frac{|(x- y)^{\perp}|^2}{|x-y|^{k+2}} \nonumber\\
\nonumber \\
 &-&2\int_{C_1\partial \Sigma} \langle \overrightarrow{H}_{C_1\partial\Sigma},W_{y}\rangle\, d_{C_1\partial \Sigma} - 2|D^k|.
	\end{eqnarray}
	The free boundary condition of $\Sigma$ combined with the Divergence Theorem applied to the position vector $X=\overrightarrow{x}$ give
	\begin{eqnarray*}
		k|C_1\partial\Sigma|= |\partial\Sigma| - k\int_{C_1\partial\Sigma} \langle \overrightarrow{H}_{C_1\partial\Sigma}, x\rangle d_{\Sigma}=|\partial\Sigma|=k|\Sigma|.
	\end{eqnarray*}
The case $y\in \partial \Sigma$ is done in similar manner with minor modifications. This completes the proof of the proposition.
\end{proof}

\section{Proof of Theorem 1.4}
Following \cite{S}, we have for  every $X \in \mathcal{X}(\mathbb{R}^{n+1})$  the following expression for the second variation of area $\Sigma^k$ in the direction of $X$
	\[	\delta^2\Sigma(X,X)= \int_{\Sigma} \bigg(|D^{\perp}X|^2-|\langle A,X\rangle|^2\bigg) d_{\Sigma} + \int_{\partial\Sigma} \langle D_XX, \nu\rangle d\sigma,\]
	where $A:\mathcal{X}(\Sigma)\times \mathcal{X}(\Sigma)\rightarrow \mathcal{X}^{\perp}(\Sigma)$ is the second fundamental form of $\Sigma$.
	
	\begin{theorem}[Fraser-Schoen \cite{FS2}]\label{index}	If $\Sigma^k$ is a free boundary minimal surface in $B^n$ and $v\in \mathbb{R}^n$, then  
		\[
		\delta^2\Sigma(v^{\perp},v^{\perp})= -k\int_{\Sigma} |v^{\perp}|^2\, d\Sigma
		\]If $\Sigma$ is not contained in a cylinder $\Sigma_0\times \mathbb{R}$ where $\Sigma_0$ is a free boundary minimal surface, then $index(\Sigma)$ is at least $n$. In particular, if $k = 2$ and $\Sigma$ is not a
		plane disk, its index is at least n.
	\end{theorem}

Recall the $p-$th eigenvalue of $\delta^2\Sigma$ has the min-max characterization:
\[
\lambda_p= \inf_{\mathcal{W}\,:\, dim(\mathcal{W})=p} \,\,\sup_{X\in \mathcal{W}}\,\, \frac{\delta^2\Sigma(X,X)}{ \int_{\Sigma} |X|^2 d\Sigma}
\]

\begin{proof}[Proof of Theorem \ref{high dimension}]
\noindent	Arguing by contradiction, we assume that $\{\Sigma_i\}$  is a sequence of non-totally geodesic $k$-dimensional  free boundary minimal surfaces  in $B^n$ satisfying 
	\begin{eqnarray}\label{gap area}
	|\Sigma_i|\rightarrow  |D^k|
	\end{eqnarray}
	Following the strategy in \cite{K}, we  first show that (\ref{gap area}) implies curvature estimates for $\Sigma_i$.
	\begin{claim}\label{gap curvature estimates}
		Let $A_{\Sigma_i}$ be the second fundamental form of $\Sigma_i$. Then, there exists  $C>0$ such that
		\begin{eqnarray}
		\sup_{x \in\Sigma_i}\,|A_{\Sigma_i}(x)|\leq C.
		\end{eqnarray}
	\end{claim}
Let us show  first that the index estimate in  Theorem \ref{index} combined with  the Claim \ref{gap curvature estimates}  imply the theorem:

	By the Claim \ref{gap curvature estimates},   the second fundamental form of   $\{\Sigma_i\}$ is uniformly bounded. Theorem 6.1 in \cite{LZ} (see also \cite{ADN}) implies that $\Sigma_i$ converges smoothly up to the boundary to a  free boundary minimal surface $\Sigma_{\infty}$. Since  $|\Sigma_{\infty}|=|D^k|$, we conclude by   Theorem  \ref{brendle} that $\Sigma_{\infty}$ is an equatorial disk. In particular, $\Sigma_i$ is diffeomorphic to a disk and  has trivial normal bundle for $i$ large enough. By  Lemma \ref{index}, $\text{index}(\Sigma_i)\geq n$ unless $\Sigma_i$ is contained in the cylinder $\overline{\Sigma}_i\times \mathbb{R}$. By our assumptions $\overline{\Sigma}_i$ is a free boundary minimal surface smooth perturbation of an equatorial disk $\mathbb{D}^{k-1}$. This suggest applying an induction argument. Note that the first step corresponds to  two-dimensional minimal surface $\Sigma_i$ in $B^n$, and by Lemma \ref{index}, $index(\Sigma_i)$ is at least $n$ unless totally geodesic. The discussion for both the case $k=2$ and  the induction step $k=n+1$  are, hence, the same. Therefore, without loss of generality, it suffices assuming that $index(\Sigma_i)\geq n$ unless $\Sigma_i$ is an equatorial disk $\mathbb{D}^k$. This implies that  there exist at least $n$ mutually orthonormal eigenvector fields $X_j$ 
of the quadratic form $\delta^2\Sigma(\cdot,\cdot)$ defined on $\mathcal{X}^{\perp}(\Sigma_i)$ satisfying
\begin{eqnarray}\label{eigenvectors equation}\Delta^{\perp} X &+& \sum_{jl}\langle A(e_j,e_l),X\rangle A(e_j,e_l) + \lambda_X\,X = 0, \\ \nonumber\\ && (D_{\nu}X - D_{X}\nu)^{T\partial B^n}=\,0,\quad \text{and}\quad \lambda_X<0 \nonumber. \end{eqnarray}
Note that the vector fields $X$ are not necessarily equal to $v^{\perp}$ from Theorem \ref{index}. As $i\rightarrow \infty$, these eigenvectors converge to eigenvectors of the Jacobi operator on  $\Sigma_{\infty}\subset B^n$ and none of these eigenvector have eigenvalue zero since by the observation after Theorem \ref{index}, $\lambda_X<-k$. This is contradiction since $index(\Sigma_{\infty})=n-k$.
\end{proof}

	\begin{proof}[Proof of Claim \ref{gap curvature estimates}]
	Arguing by contradiction, we assume that
		\[\text{Area}(\Sigma_i)\rightarrow  |D^k|
		\quad \text{and}\quad \lambda_i=\sup_{x\in\Sigma_i}|A_i|^2(x) \rightarrow \infty.\]
		For each $i$ choose $x_i \in \Sigma_i$ with the property that $\sup_{\Sigma_i}|A_i|^2=|A_i|^2(x_i)$.
	Note that $\lim_{i\rightarrow \infty}|x_i|=1$. Indeed, the excess inequality (\ref{excess bordo}) implies that $\Sigma_i$ converges with multiplicity one  to $\mathbb{D}^k$  as a varifold. Hence, in $B^n(R)$, $0<R<1$, the surface $\Sigma_i$ satisfy $\theta(\Sigma_i,x,r)\leq 1 + \varepsilon$ for every $i$ large enough and $r$ small enough. If $\lim_{i\rightarrow \infty}|x_i|<1$, then we would get a contradiction with the smooth version of Allard's regularity theorem. 
	Now we consider the surface 
		\[\hat{\Sigma}_i= \lambda_i (\Sigma_i - x_i).\]
		One can check that $\hat{\Sigma}_i$ satisfies 
		\begin{equation}\label{curvature estimates 2}
		\sup_{x\in\hat{\Sigma}_i}|A|(x)\leq 1\quad \text{and}\quad |A_{\hat{\Sigma}_i}|(0)=1
		\end{equation}
		and it is a free boundary  minimal surface  in $\lambda_i(B_1^{n+1}(0)-x_i)$. 
		 It follows from Theorem 6.1 in \cite{LZ}(see also \cite{ADN}) that, after passing to a subsequence,  $\hat{\Sigma}_i$ converges smoothly and locally uniformly to $\Sigma_{\infty}$.  $\Sigma_{\infty}$  is  either  complete  without boundary minimal surface or it is a free boundary minimal surface in a half space.  Moreover, (\ref{curvature estimates 2})  implies that
		\begin{eqnarray}\label{curvature estimates 3}
		|A_{\Sigma_{\infty}}|(0)=1.
		\end{eqnarray}
		 On the other hand, by the scale invariance of the excess, we have that
		\begin{eqnarray}\label{excess interior}
		\int_{\Sigma_\infty}\frac{|z^{\perp}|^2}{|z|^{k+2}}\,d_{\Sigma_{\infty}}\leq \liminf_{i\rightarrow \infty} \int_{\Sigma_i} \frac{|z^{\perp}|^2}{|z|^{k+2}}\,d_{\Sigma_i}=\liminf_{i\rightarrow \infty} \int_{\Sigma_i} \frac{|(x-x_i)^{\perp}|^2}{|x-x_i|^{k+2}}\,d_{\Sigma_i} \nonumber
		\end{eqnarray}
	
		We want to prove that the last term above goes to zero as $i\rightarrow \infty$.

		\begin{claim}\label{claim3}	There exist $C>0$ such that for every $y \in \partial\Sigma_i$
					\begin{eqnarray*}
						|A_{\partial\Sigma_i}|(y)\leq C
					\end{eqnarray*}
				\end{claim}

				\begin{proof} 
					 Let $w_i\in  \partial\Sigma_i$ such that $\sup_{w\in\partial\Sigma_i}|A_{C_1\partial\Sigma_i}|(w)=|A_{C_1\partial\Sigma_i}(w_i)|=\beta_i$. Take the sequence $\widetilde{\Sigma}_i=\beta_i\Sigma_i^{\prime}$, where $\Sigma_i^{\prime}= (C_1\partial\Sigma_i-w_i)$. By the cone excess  (\ref{excess cone}) and the excess (\ref{excess bordo}):
					 \begin{eqnarray*}
					 \int_{\Sigma_i^{\prime}}\frac{|(x-w_i)^{\perp}|^2}{|x-w_i|^{k+2}} 
					  +  \int_{\Sigma_i^{\prime}} \langle H, \frac{x-w_i}{|x-w_i|^k}\rangle	=  
					    \int_{\Sigma_i}\frac{|(x-w_i)^{\perp}|^2}{|x-w_i|^{k+2}} \leq \frac{|\Sigma_i|- |D^k|}{2}
					 \end{eqnarray*}
				 By compactness, we obtain that $\widetilde{\Sigma}_i$ converge $C^{1,\alpha}$ to a submanifold $V_{\infty}$ which is free boundary on a hyperplane. By the scaling invariance, we obtain
				 \[
				 \int_{V_{\infty}} \frac{| Z^{\perp}|^2}{|Z|^{k+2}} + \int_{V_{\infty}} \langle \overrightarrow{H}, \frac{Z}{|Z|^{k}}\rangle=0
				 \]  
				 Applying a reflection symmetry  with respect to the hyperplane and the Classical Monotonicity formula for varifolds with bounded mean curvature, see  \cite[Theorem 2.1]{DL} we obtain that the density at the origin is $\lim_{r\rightarrow +\infty}\theta(V_{\infty},0,r)=1$. Therefore, $V_{\infty}$  is totally geodesic. This contradicts the choice of $w_i\in \Sigma_i$.
				\end{proof}
\noindent In particular, it follows from the Claim \ref{claim3} that the second fundamental form of $\partial\Sigma_i$ in $\mathbb{R}^{n+1}$ is uniformly bounded. Thus, up to subsequence,  $\partial \Sigma_i$ converges   in the $C^{1,\alpha}$ topology to  $\partial D^k \subset \mathbb{S}^{n-1}$. Equivalently, $C_1\partial\Sigma_i\cap \big(B^n-B_r(0)\big)$ converges in the $C^{1,\alpha}$ topology to $D^k$ .
			Without loss of generality, we can  assume that $x_i\notin C_1\partial\Sigma_i$ since $\Sigma_i$ cannot coincide with $C_1\partial \Sigma_i$ near $x_i$. 
	Applying (\ref{excess brendle}) in Proposition \ref{main result}, we obtain 
		\begin{eqnarray*}
				\int_{\Sigma_{\infty}} \frac{|z^{\perp}|^2}{|z|^{k+2}} \,d_{\Sigma_{\infty}}|&\leq&  \nonumber\\ 
				\nonumber \\
				 \lim_{i\rightarrow \infty} \int_{C_1\partial\Sigma_i} \frac{|(x-x_i)^{\perp}|^2}{|x-x_i|^{k+2}}\,d_{C\partial\Sigma_i} 
				 &+& \int_{C\partial\Sigma_i}\langle \overrightarrow{H}_{C_1\partial\Sigma_i}, \frac{x-x_i}{|x-x_i|^k}\rangle\,d_{C\partial\Sigma_i} - |D^k|. \nonumber
		\end{eqnarray*}

Now we explore that $C_1\partial \Sigma_i$ converges graphically to $D^k$ to prove the following claim:
			\begin{claim}\label{main claim}
			\begin{eqnarray*}\label{beta}
			 \lim_{i\rightarrow \infty}\bigg(\int_{C_1\partial\Sigma_i} \frac{|(x-x_i)^{\perp}|^2}{|x-x_i|^{k+2}} 
			+ \int_{C\partial\Sigma_i}\langle \overrightarrow{H}_{C_1\partial\Sigma_i}, \frac{x-x_i}{|x-x_i|^k}\rangle\bigg) \leq  |D^k|.
			\end{eqnarray*}
		\end{claim}
		\begin{proof}
Since $\Sigma_i$ converges weakly to $\mathbb{D}^k$, we    assume that $x_i\rightarrow y\in \partial \mathbb{D}^k$.		First note that 
			\begin{eqnarray*}
			\lim_{i\rightarrow \infty}	\int_{C_1\partial\Sigma_i-B_s(y)} \frac{|(x-x_i)^{\perp}|^2}{|x-x_i|^{k+2}} + \int_{C\partial\Sigma_i-B_s(y)}\langle \overrightarrow{H}_{C_1\partial\Sigma_i}, \frac{x-x_i}{|x-x_i|^k}\rangle =0
			\end{eqnarray*}
	since $C_1\partial\Sigma_i\rightarrow D^k$ in the $C^{1,\alpha}$ topology, here  $B_s(y)$ denotes  an    Euclidean ball.
	Hence, it is enough to focus on $\Sigma_i\cap B_s(y)$.  The convergence $C_1\partial\Sigma_i\rightarrow D^k$ also implies that we can choose $s<1$ very small so that $T_xC_1\partial\Sigma_i
	$ is uniformly close to $T_yD^k$ for every $x\in C_1\partial\Sigma_i\cap B_s(y)$. In particular, we can write $C_1\partial\Sigma_i$ as a graph over $T_{y}D^k$ and we have that $dvol_{C_1\partial \Sigma_i}= dvol_{D^k}(1+o_i(1))$.
		Let  $z_i\in T_xC_1\partial\Sigma_i$  a point which realizes the distance $r_i=d(T_xC_1\partial\Sigma_i,x_i)$.
		Hence,  $ |(x-x_i)^{\perp}|^2= r_i^2$ and $x-x_i= x-z_i + u_i$ for every $x \in C_1\partial\Sigma_i\cap B_s(x_i)$ where $u_i \perp T_{x}C_1\partial \Sigma_i$ and $|u_i|=r_i$. In particular, $u_i\perp(x-z_i)$. 
			Therefore,
			\begin{eqnarray*}
				\lim_{i\rightarrow\infty}\int_{C_1\partial\Sigma_i \cap B_s(y)} \frac{|(x- x_i)^{\perp}|^2}{|x-x_i|^{k+2}}  =
				\lim_{i\rightarrow \infty}\int_{C_1\partial\Sigma_i\cap B_s(y)} \frac{r_i^2}{|x-z_i +u_i|^{k+2}}=   \\ 
				\\
				\lim_{i\rightarrow \infty}\int_{C_1\partial\Sigma_i\cap B_s(y)-z_i}\frac{r_i^2 r_i^{-k-2}}{|\frac{x-z_i}{r_i} + \frac{u_i}{r_i}|^{k+2}}
		= \lim_{i\rightarrow \infty}\int_{\frac{1}{r_i}(C_1\partial\Sigma_i\cap B_s(y)-z_i)}\frac{1}{(|w|^2+1)^{\frac{k+2}{2}}}\\
		\\
					=
			\int_{P_1} \frac{1}{(|y|^2+1)^{\frac{k+2}{2}}}  	\leq 
			 \int_{\mathbb{R}^k} \frac{1}{(|y|^2+1)^{\frac{k+2}{2}}}
				= \int_{0}^{\infty}\int_{\partial D^k} \frac{s^{k-1}}{(s^2+1)^{\frac{k+2}{2}}} ds \\
				\\
				= |\partial D^k| \int_{0}^{\infty}\frac{s^{k-1}}{(s^2+1)^{\frac{k+2}{2}}}\, ds= \frac{|\partial D^k|}{k}\,=\, |D^k| ,
			\end{eqnarray*}
			where $P_1$ is either  $\mathbb{R}^k$ or a half space $\mathbb{R}_{a}^k=\{x\in \mathbb{R}^k\,:\,\langle x,e_k\rangle \leq a\}$. Similarly,
			\begin{eqnarray*}
				&&\lim_{i\rightarrow \infty} \int_{C_1\partial\Sigma_i\cap B_s(y)}
				\langle \overrightarrow{H}_{C_1\partial\Sigma_i}, \frac{(x-x_i)}{|x-x_i|^k}d_{C_1\partial\Sigma_i}\, \leq \\
				\\
				&&  \lim_{i\rightarrow \infty} \sup_{C_1\partial\Sigma_i\cap B_s(y)} \left|\left< \overrightarrow{H}_{C_1\partial\Sigma_i},\frac{(x-x_i)}{|x-x_i|}\right>\right| \int_{C_1\partial\Sigma_i\cap B_s(y)}\frac{1}{|x-x_i|^{k-1}}d_{C_1\partial\Sigma_i}
			\end{eqnarray*}
		\begin{eqnarray*}
			&=& \lim_{i\rightarrow \infty}  O(1-s)\int_{C_1\partial\Sigma_i\cap B_s(y)}\frac{1}{|x-z_i + u_i|^{k-1}}\, d_{C_1\partial\Sigma_i}  \\
			\\
			&=&\lim_{i\rightarrow \infty} O(1-s)\,\int_{C_1\partial \Sigma_i\cap B_s(y)-z_i} \frac{r_i^{1-k}}{|\frac{x-z_i}{r_i} + \frac{u_i}{r_i}|^{k-1}}\, d_{C_1\partial\Sigma_i} \\
			\\
			&=&\lim_{i\rightarrow \infty} O(1-s)\,\int_{\frac{1}{r_i}(T_{z_i}C\partial \Sigma_i\cap B_s(y)-z_i)} \frac{r_i}{|w+ v_i|^{k-1}}\,d\mathcal{H}^k  \quad \quad (|v_i|=1)\\
			\\
			&\leq& \lim_{i\rightarrow \infty} O(1-s) \int_{\mathbb{R}^k \cap B_{\frac{s}{r_i}}}\frac{r_i}{|w+ v_i|^{k-1}}\,dw \,(1+o_i(1)) \\
			\\
			&\leq& \lim_{i\rightarrow \infty} O(1-s) \int_{\mathbb{R}^k \cap B_{\frac{s}{r_i}}}\frac{r_i}{|w|^{k-1}}\, dw \,(1+o_i(1)) \\
		&\leq& \lim_{i\rightarrow \infty} O(1-s)\int_{0}^{\frac{s}{r_i}} r_i\,\frac{s^{k-1}}{s^{k-1}} ds\,(1+o_i(1)) =  O(s).
			\end{eqnarray*}
	We used that $u_i\perp (x-z_i)$ to obtain $|w|\leq |w+v_i|$ above.	Making $s\rightarrow 0$, we conclude the proof of the Claim \ref{main claim}.
		\end{proof}
\noindent Using the  Claim above, we obtain that
\begin{eqnarray*}
	\int_{\Sigma_{\infty}} \frac{|z^{\perp}|^2}{|z|^{k+2}} \,d_{\Sigma_{\infty}}\leq |D^k|-|D^k|=0.
\end{eqnarray*}
 Since this contradicts  (\ref{curvature estimates 3}), the Claim \ref{gap curvature estimates} is proved. 
\end{proof}

\begin{proof}[Proof of Proposition \ref{unique continuation}]
	Let $C_1\partial \Sigma$ be the free boundary cone over the boundary $\partial \Sigma$. Since $\partial \Sigma$ is minimal in $\mathbb{S}^{n-1}$, $C_1\partial \Sigma$  is a minimal submanifold in $\mathbb{R}^n$. Moreover, $\Sigma$ and $C_1\partial \Sigma$ coincide up to first order at $\partial \Sigma$. More precisely, if we write locally $\Sigma=graph(u)$ and $C_1\partial\Sigma=graph(v)$  near $x\in\partial\Sigma$, where $u,v: U\subset T_x\Sigma \rightarrow \mathbb{R}^{n-k}$, then  $w=v-u$ satisfies $w=\nabla w=0$ at $\partial \Sigma$. Moreover, $w$ satisfies, for each $l=1,\ldots,n-k$, the   second order linear elliptic system of  equations:\begin{eqnarray*}	\frac{a_{ij}(\nabla u)}{\sqrt{1+ |\nabla u_l|^2}}D_{ij}(w_l) + \sum_{m=1}^{p}b_j^m(\nabla u,\nabla v) D_{j} (w_m)=0,\end{eqnarray*}for some smooth functions $a_{ij}(\nabla u)$ and $b_j^m(\nabla u,\nabla v)$. By a result of Simon-Hardt \cite{HS} (see also \cite[Lemma 2.3]{BV} for the  generalization to systems of equations), $w$ vanishes at infinite order at $\partial\Sigma$. On the other hand,  \cite[Lemma 1]{L} and Morrey \cite[Theorem 6.7.6]{M}  imply that  the map $w$ is analytic. Therefore, $w\equiv 0$ and  $\Sigma=C_1\partial\Sigma$.\end{proof}

\end{document}